\documentclass[12pt,a4paper,oneside]{article}
\usepackage[margin=1in]{geometry}
\usepackage{amsmath,amssymb,amsthm,mathtools}
\usepackage[bookmarks=true,colorlinks=true,citecolor=blue,linkcolor=black,urlcolor=blue]{hyperref}
\numberwithin{equation}{section}

\theoremstyle{plain}
\newtheorem{theorem}{Theorem}[section]
\newtheorem{lemma}[theorem]{Lemma}
\newtheorem{proposition}[theorem]{Proposition}
\theoremstyle{definition}
\newtheorem{definition}[theorem]{Definition}
\newcommand{\T}{\mathcal T}
\newcommand{\R}{\mathbb R}

\newcommand{\wt}{\widetilde}
\DeclareMathOperator{\dist}{dist}

\title{Infinite Combinatorial Yamabe Flows in Three Dimensions}
\author{Bohao Ji}
\date{}

\begin{document}
\maketitle

\begin{abstract}
In this paper, we study three-dimensional combinatorial Yamabe flows on locally finite infinite triangulations in Euclidean and hyperbolic background geometries.  Under suitable non-degeneracy and bounded-degree assumptions, we establish the short-time existence and uniqueness of the original flows.  We further introduce the extended flows by using the continuous extension of solid angles, and prove long-time existence for both extended flows under suitable initial assumptions.
\end{abstract}

\section{Introduction}
The Yamabe problem asks whether a Riemannian metric can be conformally deformed to one with constant scalar curvature.  Its resolution in the smooth compact case, due to Yamabe, Trudinger, Aubin, and Schoen, is a central result in conformal geometry; see \cite{Yamabe1960,Schoen1984}.  The parabolic approach to curvature prescription, represented by Hamilton's work on curvature flows \cite{Hamilton1988}, has an equally important discrete counterpart.  In discrete conformal geometry, smooth metrics are replaced by metrics on triangulated or polyhedral spaces, and curvature is represented by angle defects.

Circle packing theory initiated by Thurston \cite{Thurston1980} and developed by Rodin--Sullivan \cite{RodinSullivan1987}, He--Schramm \cite{HeSchramm1995}, and He \cite{He1999} provides a fundamental discrete analogue of conformal geometry on surfaces.  On finite triangulated surfaces, Chow and Luo introduced combinatorial Ricci flows for circle packing metrics \cite{ChowLuo2003}, and Luo introduced combinatorial Yamabe flow for PL metrics \cite{Luo2004}.  Later discrete uniformization results were obtained by Gu, Luo, Sun, and Wu \cite{GuLuoSunWu2018,GuGuoLuoSunWu2018}, and convergence to the classical theory was studied by Luo, Sun, and Wu \cite{LuoSunWu2022}.  Recent work of Ge, Hua, and Zhou on combinatorial Ricci flows on infinite disk triangulations \cite{GeHuaZhou2025} shows that finite exhaustion and maximum-principle arguments can be used effectively in this noncompact discrete setting.

In dimension three, Cooper and Rivin introduced ball-packing metrics and a combinatorial scalar curvature given by solid-angle defects at vertices \cite{CooperRivin1996}.  Glickenstein defined the Euclidean three-dimensional combinatorial Yamabe flow, derived the curvature evolution equation, and identified the associated dual-area Laplacian \cite{Glickenstein2005,Glickenstein2011}.  Ge, Jiang, and Shen studied extended Euclidean ball-packing flows through finite-time degenerations \cite{GeJiangShen2022}.  In hyperbolic background geometry, Ge and Hua introduced the corresponding three-dimensional Yamabe flow and proved finite-dimensional long-time existence, uniqueness, and convergence or non-existence results related to tetra-degree conditions \cite{GeHua2020}.  These results concern finite triangulations; the goal of this paper is to establish basic well-posedness and extended-flow existence results for locally finite infinite triangulations.

We now describe the objects studied in this paper.  Let $\T=(\T_0,\T_1,\T_2,\T_3)$ be a locally finite triangulation of a three-manifold without boundary.  For a vertex $i$, let $q_i$ be the graph degree of $i$, namely the number of adjacent vertices, and let $d_i$ be the number of tetrahedra incident to $i$.  We say that $\T$ has bounded degree if $\sup_{i\in\T_0}q_i<\infty$.  This also gives a uniform bound for $d_i$.  A ball-packing metric is a function
\[
 r:\T_0\to(0,\infty),
\]
and each edge $ij\in\T_1$ is assigned length $l_{ij}=r_i+r_j$.  A tetrahedron is called real if these six edge lengths are realized by a non-degenerate geodesic tetrahedron in the background geometry; otherwise it is called virtual.  For a Euclidean tetrahedron $\{i,j,k,l\}$, set $x_a=1/r_a$.  The tetrahedron is real precisely when
\[
 Q^E_{ijkl}:=(x_i+x_j+x_k+x_l)^2-2(x_i^2+x_j^2+x_k^2+x_l^2)>0.
\]
For a hyperbolic tetrahedron, set $c_a=\coth r_a$; the corresponding real condition is
\[
 Q^H_{ijkl}:=(c_i+c_j+c_k+c_l)^2-2(c_i^2+c_j^2+c_k^2+c_l^2)+4>0.
\]
If all incident tetrahedra are real, the Cooper--Rivin curvature at a vertex $i$ is
\[
 K_i=4\pi-\sum_{\{i,j,k,l\}\in\T_3}\alpha_{ijkl},
\]
where $\alpha_{ijkl}$ is the solid angle at $i$ in the tetrahedron $\{i,j,k,l\}$.  The Euclidean and hyperbolic combinatorial Yamabe flows are respectively
\[
 \frac{dr_i}{dt}=-K_i r_i,\qquad
 \frac{dr_i}{dt}=-K_i\sinh r_i.
\]
Equivalently, in the conformal variables $u_i=\log r_i$ and $w_i=\log\tanh(r_i/2)$, both original flows take the form $u_i'=-K_i$ or $w_i'=-K_i$.

\begin{definition}[Strong non-degenerate conditions]\label{def:intro-strong}
Let $r$ be a ball-packing metric on $\T$.  For a vertex $a$ of a tetrahedron, write $x_a=1/r_a$ in Euclidean background geometry and $c_a=\coth r_a$ in hyperbolic background geometry.
\begin{enumerate}
\item In Euclidean background geometry, $r$ satisfies the strong non-degenerate condition I if, for every ordered tetrahedron $\{i,j,k,l\}\in\T_3$,
\[
 x_i\le x_j+x_k+x_l .
\]
It satisfies the strong non-degenerate condition II if there exist constants $\varepsilon,C>0$ such that, for every ordered tetrahedron $\{i,j,k,l\}\in\T_3$,
\[
 \varepsilon\le x_j+x_k+x_l-x_i\le C.
\]
\item In hyperbolic background geometry, $r$ satisfies the strong non-degenerate condition I if, for every ordered tetrahedron $\{i,j,k,l\}\in\T_3$,
\[
 c_i\le c_j+c_k+c_l .
\]
\end{enumerate}
\end{definition}

We divide the main results according to the three analytic questions treated in the paper.

\begin{theorem}[Short-time existence]\label{thm:intro-short}
Let $\T$ be a locally finite triangulation with bounded degree.
\begin{enumerate}
\renewcommand{\labelenumi}{(\arabic{enumi})}
\item In Euclidean background geometry, let $r^0:\T_0\to(0,\infty)$ be real and assume that there exists $\rho>0$ such that $e^v r^0$ is real for every $v:\T_0\to\R$ with $\|v\|_{\ell^\infty}\le\rho$.  Then there exist $T_E>0$ and a solution $r(t)$ of
\[
 \frac{dr_i}{dt}=-K_i r_i,\qquad r_i(0)=r_i^0,
\]
for all $i\in\T_0$ and $t\in[0,T_E)$.
\item In hyperbolic background geometry, let $r^0:\T_0\to(0,\infty)$ and set $w_i^0=\log\tanh(r_i^0/2)$.  Assume that there exists $\rho>0$ such that, for every $v:\T_0\to\R$ with $\|v\|_{\ell^\infty}\le\rho$, one has $w_i^0+v_i<0$ for all $i$ and the metric
\[
 r_i(v)=2\operatorname{arctanh}\bigl(e^{w_i^0+v_i}\bigr)
\]
is real.  Then there exist $T_H>0$ and a solution $r(t)$ of
\[
 \frac{dr_i}{dt}=-K_i\sinh r_i,\qquad r_i(0)=r_i^0,
\]
for all $i\in\T_0$ and $t\in[0,T_H)$.
\end{enumerate}
\end{theorem}

\begin{theorem}[Uniqueness]\label{thm:intro-unique}
Let $\T$ be a locally finite triangulation with bounded degree.
\begin{enumerate}
\renewcommand{\labelenumi}{(\arabic{enumi})}
\item Let $r(t)$ and $\hat r(t)$ be two Euclidean solutions on $[0,T]$ with the same initial data.  If, for every $s\in[0,1]$, the logarithmic interpolation
\[
 u^s(t)=s\log r(t)+(1-s)\log\hat r(t)
\]
defines a metric $r^s(t)$ by $r_i^s(t)=e^{u_i^s(t)}$ satisfying the strong non-degenerate condition II uniformly on $[0,T]$, then $r(t)\equiv\hat r(t)$.
\item Let $r(t)$ and $\hat r(t)$ be two hyperbolic solutions on $[0,T]$ with the same initial data.  Put $w_i(t)=\log\tanh(r_i(t)/2)$ and $\hat w_i(t)=\log\tanh(\hat r_i(t)/2)$.  For $s\in[0,1]$, let
\[
 w^s(t)=s w(t)+(1-s)\hat w(t)
\]
and define $r_i^s(t)=2\operatorname{arctanh}(e^{w_i^s(t)})$.  Suppose that there exists $r_*>0$ such that $r_i^s(t)\ge r_*$ and $r^s(t)$ satisfies the hyperbolic strong non-degenerate condition I for every $s\in[0,1]$ and $t\in[0,T]$.  Then $r(t)\equiv\hat r(t)$.
\end{enumerate}
\end{theorem}

To continue the flow after a tetrahedron becomes degenerate, one replaces the ordinary solid angles by their continuous extensions.  In the Euclidean case, when a virtual tetrahedron has one dominant vertex, the extended solid angle at that vertex is set to be $2\pi$ and the other three extended solid angles are set to be $0$; the hyperbolic case admits an analogous continuous extension.  Thus the extended curvature $\widetilde K_i$ is defined for arbitrary positive radii, including virtual configurations.

\begin{theorem}[Existence of extended flows]\label{thm:intro-extended}
Let $\widetilde\alpha$ denote the continuous extension of solid angles to real and virtual tetrahedra, and set $\widetilde K_i=4\pi-\sum\widetilde\alpha_{ijkl}$.
\begin{enumerate}
\renewcommand{\labelenumi}{(\arabic{enumi})}
\item For every locally finite $\T$ and every Euclidean initial metric $r^0:\T_0\to(0,\infty)$, there exists a global solution of
\[
 \frac{dr_i}{dt}=-\widetilde K_i r_i,\qquad r_i(0)=r_i^0,
\]
with $r_i\in C^1([0,\infty))$ for every $i\in\T_0$.
\item In hyperbolic background geometry, assume that $\T$ has bounded degree and that $\sup_i r_i^0<\infty$.  Then there exists a global solution of
\[
 \frac{dr_i}{dt}=-\widetilde K_i\sinh r_i,\qquad r_i(0)=r_i^0,
\]
with $r_i\in C^1([0,\infty))$ for every $i\in\T_0$.
\end{enumerate}
\end{theorem}

The paper is organized as follows.  Section 2 fixes the notation for combinatorial Yamabe flows, the graph maximum principle, and the extension of solid angles and flows.  Section 3 treats the Euclidean background geometry: non-degeneracy, curvature evolution, short-time existence, and uniqueness.  Section 4 proves the corresponding hyperbolic short-time existence and uniqueness results.  Section 5 proves the long-time existence of the extended flows.

\section{Preliminaries}
\subsection{Combinatorial Yamabe Flows and Maximum Principles}
Let $\T=(\T_0,\T_1,\T_2,\T_3)$ be a locally finite triangulation of a three-manifold without boundary, where $\T_0,\T_1,\T_2,\T_3$ denote the sets of vertices, edges, faces, and tetrahedra.  We write $i\sim j$ if $ij\in\T_1$.  Local finiteness means that every vertex has only finitely many adjacent vertices and is incident to only finitely many tetrahedra.

For $i\in\T_0$, let
\[
 q_i=\#\{j\in\T_0:j\sim i\},\qquad
 d_i=\#\{\{i,j,k,l\}\in\T_3\}.
\]
Here $q_i$ is the usual graph degree of the vertex $i$.  We say that $\T$ has degree bounded by $M$ if $q_i\le M$ for all $i\in\T_0$.  Since a tetrahedron incident to $i$ is determined by three adjacent vertices of $i$, this implies
\[
 d_i\le D_M:=\binom{M}{3}.
\]
A ball-packing metric is a map $r:\T_0\to(0,\infty)$, and each edge has length $l_{ij}=r_i+r_j$.

Whenever all tetrahedra incident to $i$ are real, let $\alpha_{ijkl}$ be the solid angle at $i$ in the tetrahedron $\{i,j,k,l\}$ and set
\[
 K_i=4\pi-\sum_{\{i,j,k,l\}\in\T_3}\alpha_{ijkl}.
\]
This is the Cooper--Rivin combinatorial scalar curvature \cite{CooperRivin1996}.  The Euclidean three-dimensional combinatorial Yamabe flow introduced by Glickenstein \cite{Glickenstein2005} and its hyperbolic analogue studied by Ge and Hua \cite{GeHua2020} are
\[
 \frac{dr_i}{dt}=-K_i r_i,\qquad
 \frac{dr_i}{dt}=-K_i\sinh r_i.
\]
For Euclidean metrics we use the conformal variable $u_i=\log r_i$, while for hyperbolic metrics we use $w_i=\log\tanh(r_i/2)$; in these variables the corresponding original flow is $u_i'=-K_i$ or $w_i'=-K_i$.

We use pointwise solutions throughout.  Thus a family $r(t):\T_0\to(0,\infty)$ is a $C^1$ solution on an interval $I$ if $r_i\in C^1(I)$ for every $i\in\T_0$ and the relevant differential equation holds at each vertex.

\begin{lemma}[Elementary curvature bound]\label{lem:curv-bound}
Assume that each solid angle lies in $[0,2\pi]$.  Then
\[
 |K_i|\le 4\pi+2\pi d_i.
\]
In particular, if $\T$ has degree bounded by $M$, then $|K_i|$ is bounded by a constant depending only on $M$.
\end{lemma}
\begin{proof}
For a fixed vertex $i$, the curvature is $4\pi$ minus a finite sum over the tetrahedra incident to $i$.  The sum is between $0$ and $2\pi d_i$, which gives the estimate.
\end{proof}

For a locally finite weighted graph $G=(V,E)$ with time-dependent weights, set
\[
 (\Delta f)_i=\frac1{a_i(t)}\sum_{j:j\sim i}\omega_{ij}(t)(f_j-f_i),\qquad \omega_{ij}(t)\ge0,\quad a_i(t)>0.
\]
The weights may depend on time.  We use the bounded weighted-degree condition
\[
 \sup_{(i,t)\in V\times[0,T]}\frac1{a_i(t)}\sum_{j:j\sim i}\omega_{ij}(t)<\infty.
\]

The following maximum principle is a weighted form of the infinite-graph maximum principle used by Ge, Hua, and Zhou \cite{GeHuaZhou2025}.

\begin{lemma}[Maximum principle on infinite graphs]\label{lem:mp}
Assume
\[
 \frac1{a_i(t)}\sum_{j:j\sim i}\omega_{ij}(t)\le C_0,\qquad h_i(t)\le C_1
\]
on $V\times[0,T]$.  If $f$ is bounded and
\[
 f_i'\le (\Delta f)_i+h_i f_i,
\]
then $f\le0$ on $V\times[0,T]$.  Hence a bounded solution of $f'=
\Delta f+hf$ with $f(0)=0$ is identically zero.
\end{lemma}
\begin{proof}
It suffices to work on one connected component of $G$.  Choose $C>C_1$ and put $F_i=e^{-Ct}f_i$.  Then
\[
 F_i'\le(\Delta F)_i+(h_i-C)F_i.
\]
Fix a vertex $o$ in this component and set $d(i)=\dist(o,i)$.  Since $|d(j)-d(i)|\le1$ for adjacent vertices,
\[
 (\Delta d)_i\le \frac1{a_i}\sum_{j:j\sim i}\omega_{ij}\le C_0.
\]
For $\delta>0$ define
\[
 F_i^\delta=F_i-\delta d(i)-\delta(C_0+1)t.
\]
Because $F$ is bounded and the balls centered at $o$ are finite, the term $-\delta d(i)$ forces $F^\delta$ to tend to $-\infty$ along sequences escaping finite sets.  Hence $F^\delta$ attains its maximum on $V\times[0,T]$.  If this maximum were positive and attained at some $(i_0,t_0)$, then $t_0>0$ because $f(0)\le0$.  At this point, $F_{i_0}(t_0)>0$, $\partial_tF^\delta_{i_0}(t_0)\ge0$, and $(\Delta F^\delta)_{i_0}(t_0)\le0$.  On the other hand,
\[
 \partial_tF^\delta
 \le \Delta F^\delta+\delta\Delta d+(h-C)F-\delta(C_0+1)
 < \Delta F^\delta
\]
at $(i_0,t_0)$, a contradiction.  Thus $F^\delta\le0$ for every $\delta>0$.  Letting $\delta\downarrow0$ gives $f\le0$.  Applying the result to $f$ and $-f$ gives uniqueness.
\end{proof}

\subsection{Extension of the Flow}
The original flows are defined only as long as all tetrahedra remain real.  To continue a solution through degenerations, one extends the solid angles from the real configuration space to all positive radii, following the extension theory for Euclidean and hyperbolic ball packings \cite{GeJiangShen2022,GeHua2020}.

Let $\tau=\{a,b,c,d\}$ be a tetrahedron and let $a,b,c,d$ denote its four vertices.  In the Euclidean case set $x_a=1/r_a$.  If $Q^E_\tau>0$, define $\wt\alpha^E_{a,\tau}$ to be the ordinary solid angle $\alpha_{a,\tau}$.  If $Q^E_\tau\le0$, then $\tau$ belongs to a unique virtual component $D_a^E$, characterized by
\[
 x_a\ge x_b+x_c+x_d+2\sqrt{x_bx_c+x_bx_d+x_cx_d},
\]
or equivalently by the fact that the ball at $a$ passes through the gap between the other three mutually tangent balls.  In this case one sets
\[
 \wt\alpha^E_{a,\tau}=2\pi,\qquad
 \wt\alpha^E_{b,\tau}=\wt\alpha^E_{c,\tau}=\wt\alpha^E_{d,\tau}=0.
\]
This is the continuous extension of Euclidean solid angles constructed in \cite{GeJiangShen2022}.

In the hyperbolic case set $c_a=\coth r_a$.  If $Q^H_\tau>0$, define $\wt\alpha^H_{a,\tau}=\alpha_{a,\tau}$.  If $Q^H_\tau\le0$, then $\tau$ belongs to a unique virtual component $D_a^H$, characterized by
\[
 c_a\ge c_b+c_c+c_d+2\sqrt{c_bc_c+c_bc_d+c_cc_d+1}.
\]
For $r_\tau\in D_a^H$ one defines
\[
 \wt\alpha^H_{a,\tau}=2\pi,\qquad
 \wt\alpha^H_{b,\tau}=\wt\alpha^H_{c,\tau}=\wt\alpha^H_{d,\tau}=0.
\]
This is the continuous extension of hyperbolic solid angles constructed in \cite{GeHua2020}.

The extended Cooper--Rivin curvature is therefore defined for every positive ball-packing metric by
\[
 \wt K_i^E=4\pi-\sum_{\tau\in\T_3:\,i\in\tau}\wt\alpha^E_{i,\tau},\qquad
 \wt K_i^H=4\pi-\sum_{\tau\in\T_3:\,i\in\tau}\wt\alpha^H_{i,\tau}.
\]
When the background geometry is fixed, we omit the superscript and write simply $\wt K_i$.  The extended Euclidean and hyperbolic combinatorial Yamabe flows are respectively
\[
 \frac{dr_i}{dt}=-\wt K_i r_i,\qquad
 \frac{dr_i}{dt}=-\wt K_i\sinh r_i.
\]

\section{Euclidean Background Geometry}
\subsection{Euclidean Ball-Packings and Curvature Evolution}
In the Euclidean background geometry, a ball-packing metric assigns the edge length $l_{ab}=r_a+r_b$ to each edge $ab$.  For a tetrahedron $\{i,j,k,l\}$, set $x_a=1/r_a$.  The six edge lengths determine a non-degenerate Euclidean tetrahedron precisely when
\begin{equation}\label{eq:QE}
 Q^E_{ijkl}=(x_i+x_j+x_k+x_l)^2-2(x_i^2+x_j^2+x_k^2+x_l^2)>0,
\end{equation}
as recalled from the finite-dimensional ball-packing theory \cite{CooperRivin1996,Glickenstein2005}.

\begin{proposition}[Strong non-degeneracy implies reality]\label{prop:E-strong-real}
Let $\tau=\{i,j,k,l\}$ be a Euclidean ball-packing tetrahedron and set $S_\tau=x_i+x_j+x_k+x_l$.  If
\[
 x_a\le S_\tau-x_a,\qquad a\in\{i,j,k,l\},
\]
then $Q^E_{ijkl}>0$, and hence $\tau$ is real.  More precisely, if there exists $\varepsilon>0$ such that
\[
 S_\tau-2x_a\ge\varepsilon,\qquad a\in\{i,j,k,l\},
\]
then $Q^E_{ijkl}\ge \varepsilon S_\tau>0$.
\end{proposition}
\begin{proof}
Set $p_a=x_a/S_\tau$.  Then $p_a>0$, $\sum_{a\in\tau}p_a=1$, and the assumptions give $p_a\le1/2$ for every $a$.  Hence
\[
 \sum_{a\in\tau}p_a^2<\frac12 .
\]
Indeed, if $p_m=\max_a p_a$, then $\sum_a p_a^2\le p_m\le1/2$; equality cannot occur because all four $p_a$ are positive and sum to one.  Therefore
\[
 Q^E_{ijkl}=S_\tau^2\left(1-2\sum_{a\in\tau}p_a^2\right)>0 .
\]
Under the stronger uniform assumption $S_\tau-2x_a\ge\varepsilon$, we have $2x_a\le S_\tau-\varepsilon$ for every $a$, and therefore
\[
 2\sum_{a\in\tau}x_a^2\le (S_\tau-\varepsilon)\sum_{a\in\tau}x_a
 =S_\tau^2-\varepsilon S_\tau.
\]
This gives $Q^E_{ijkl}\ge\varepsilon S_\tau>0$.
\end{proof}

Thus the Euclidean strong non-degenerate condition I in Definition \ref{def:intro-strong} is a sufficient non-degeneracy condition: every incident tetrahedron is real.  Uniform lower bounds for the margins $x_j+x_k+x_l-x_i$ give the quantitative estimates used later.

The Cooper--Rivin curvature is
\begin{equation}\label{eq:KE}
 K_i=4\pi-\sum_{\{i,j,k,l\}\in\T_3}\alpha_{ijkl}.
\end{equation}
The Euclidean Yamabe flow is
\begin{equation}\label{eq:Eflow}
 r_i'=-K_i r_i,
\end{equation}
or, in the variable $u_i=\log r_i$, $u_i'=-K_i$.

We next derive the Euclidean curvature evolution formula.  Set
\begin{equation}\label{eq:muE}
 \mu_{ij}=\sum_{\{i,j,k,l\}\in\T_3}\frac{\partial\alpha_{ijkl}}{\partial r_j}r_ir_j .
\end{equation}

The following identity is a standard consequence of the Euclidean Schlaefli formula in the finite-dimensional ball-packing theory \cite{Glickenstein2005}.  For completeness, we include the proof.

\begin{lemma}[Schlaefli identity for solid angles]\label{lem:E-schlafli-solid}
Let $\tau=\{a,b,c,d\}$ be a real Euclidean ball-packing tetrahedron, and let $\alpha_{a,\tau}$ be its solid angle at the vertex $a\in\tau$.  Then
\begin{equation}\label{eq:EangleEuler}
 \sum_{b\in\tau}r_b\frac{\partial\alpha_{a,\tau}}{\partial r_b}=0,\qquad a\in\tau .
\end{equation}
\end{lemma}
\begin{proof}
Let $\beta_{ab}$ be the dihedral angle at the edge $ab$ of $\tau$.  Since
\[
 \alpha_{a,\tau}=\sum_{b\in\tau\setminus\{a\}}\beta_{ab}-\pi
\]
and $l_{ab}=r_a+r_b$, the Euclidean Schlaefli formula, as used in the ball-packing setting by Glickenstein \cite{Glickenstein2005}, gives
\[
\begin{aligned}
 \sum_{a\in\tau}r_a\,d\alpha_{a,\tau}
 &=\sum_{a\in\tau}r_a\sum_{b\in\tau\setminus\{a\}}d\beta_{ab}  \\
 &=\sum_{\{a,b\}\subset\tau}(r_a+r_b)d\beta_{ab}  \\
 &=\sum_{\{a,b\}\subset\tau}l_{ab}\,d\beta_{ab}=0 .
\end{aligned}
\]
Thus the function $F_\tau=\sum_{a\in\tau}r_a\alpha_{a,\tau}$ satisfies
\[
 dF_\tau=\sum_{a\in\tau}\alpha_{a,\tau}\,dr_a+
 \sum_{a\in\tau}r_a\,d\alpha_{a,\tau}
 =\sum_{a\in\tau}\alpha_{a,\tau}\,dr_a .
\]
It follows that
\[
 \frac{\partial\alpha_{a,\tau}}{\partial r_b}
 =\frac{\partial^2F_\tau}{\partial r_b\partial r_a}
 =\frac{\partial^2F_\tau}{\partial r_a\partial r_b}
 =\frac{\partial\alpha_{b,\tau}}{\partial r_a}.
\]
Taking the coefficient of $dr_a$ in $\sum_{b\in\tau}r_b\,d\alpha_{b,\tau}=0$ gives
\[
 \sum_{b\in\tau}r_b\frac{\partial\alpha_{b,\tau}}{\partial r_a}=0.
\]
Using the symmetry of the first derivatives just proved, this becomes \eqref{eq:EangleEuler}.
\end{proof}

\begin{proposition}[Curvature evolution]\label{prop:Eevol}
Let $r(t):\T_0\to(0,\infty)$ be a smooth one-parameter family of Euclidean ball-packing metrics such that every tetrahedron is real for all times under consideration, and set $u_i(t)=\log r_i(t)$.  Then
\begin{equation}\label{eq:Evar}
 K_i'=-\frac1{r_i}\sum_{j:j\sim i}\mu_{ij}(u_j'-u_i').
\end{equation}
In particular, along \eqref{eq:Eflow},
\begin{equation}\label{eq:Eheat}
 K_i'=\frac1{r_i}\sum_{j:j\sim i}\mu_{ij}(K_j-K_i).
\end{equation}
\end{proposition}
\begin{proof}
Fix a tetrahedron $\tau=\{i,j,k,l\}$ incident to $i$.  Since $u_a(t)=\log r_a(t)$, we have $r_a'=r_a u_a'$, and the chain rule gives
\[
\begin{aligned}
 \frac{d}{dt}\alpha_{ijkl}
 &=\sum_{a\in\{i,j,k,l\}}\frac{\partial\alpha_{ijkl}}{\partial r_a}r_a u_a'  \\
 &=u_i'\sum_{a\in\{i,j,k,l\}}\frac{\partial\alpha_{ijkl}}{\partial r_a}r_a
   +\sum_{a\in\{j,k,l\}}\frac{\partial\alpha_{ijkl}}{\partial r_a}r_a(u_a'-u_i').
\end{aligned}
\]
The first term on the right-hand side vanishes by Lemma \ref{lem:E-schlafli-solid}.  Therefore
\[
\begin{aligned}
 \frac{d}{dt}\alpha_{ijkl}
 &=\sum_{a\in\{j,k,l\}}\frac{\partial\alpha_{ijkl}}{\partial r_a}r_a(u_a'-u_i')  \\
 &=\frac1{r_i}\sum_{a\in\{j,k,l\}}
   \frac{\partial\alpha_{ijkl}}{\partial r_a}r_ir_a(u_a'-u_i') .
\end{aligned}
\]
Differentiating \eqref{eq:KE} now yields
\[
\begin{aligned}
 K_i'
 &=-\sum_{\{i,j,k,l\}\in\T_3}\frac{d}{dt}\alpha_{ijkl} \\
 &=-\frac1{r_i}\sum_{\{i,j,k,l\}\in\T_3}
   \left[
   \frac{\partial\alpha_{ijkl}}{\partial r_j}r_ir_j(u_j'-u_i')
   +\frac{\partial\alpha_{ijkl}}{\partial r_k}r_ir_k(u_k'-u_i')
   +\frac{\partial\alpha_{ijkl}}{\partial r_l}r_ir_l(u_l'-u_i')
   \right].
\end{aligned}
\]
We then group all summands according to the vertex paired with $i$.  The coefficient of $(u_j'-u_i')$ is exactly
\[
 \sum_{\{i,j,k,l\}\in\T_3}\frac{\partial\alpha_{ijkl}}{\partial r_j}r_ir_j=\mu_{ij},
\]
where the sum is over tetrahedra containing the edge $ij$.  Hence
\[
 K_i'=-\frac1{r_i}\sum_{j:j\sim i}\mu_{ij}(u_j'-u_i'),
\]
which proves \eqref{eq:Evar}.  Along the Euclidean Yamabe flow, $u_a'=-K_a$ for every vertex $a$, so
\[
 u_j'-u_i'=(-K_j)-(-K_i)=K_i-K_j.
\]
Substituting this into \eqref{eq:Evar} gives
\[
 K_i'=-\frac1{r_i}\sum_{j:j\sim i}\mu_{ij}(K_i-K_j)
 =\frac1{r_i}\sum_{j:j\sim i}\mu_{ij}(K_j-K_i),
\]
which is \eqref{eq:Eheat}, the Euclidean curvature evolution equation for the combinatorial Yamabe flow.
\end{proof}

We now recall the geometric meaning of the coefficient $\mu_{ij}$ in terms of the dual area associated with the edge $ij$.  The circumscripted sphere of a Euclidean tetrahedron means the sphere tangent to all six edges of the tetrahedron.  For a labeled Euclidean tetrahedron, the existence of such a sphere is equivalent to the circumscriptible edge condition: the edge lengths can be written as $l_{ab}=r_a+r_b$ for positive vertex parameters.  Equivalently,
\[
 l_{ij}+l_{kl}=l_{ik}+l_{jl}=l_{il}+l_{jk},
\]
with the resulting vertex parameters positive.  Thus this condition is built into the Euclidean ball-packing metrics considered here.  The remaining real, non-degenerate condition is $Q^E_{ijkl}>0$; in that case the circumscripted sphere has radius $\rho_{ijkl}$ satisfying $Q^E_{ijkl}=4/\rho_{ijkl}^2$; see \cite{Glickenstein2005}.

Let $C_{ijkl}$ be the center of the circumscripted sphere of the real tetrahedron $\{i,j,k,l\}$.  Let $r_{ijk}$ be the inradius of the triangular face $\{i,j,k\}$.  We write $h_{ijk,l}$ for the signed distance from $C_{ijkl}$ to the plane of $\{i,j,k\}$, positive when $C_{ijkl}$ and the opposite vertex $l$ lie on the same side of that plane and negative otherwise.  The signed dual-area contribution associated with the edge $ij$ in this tetrahedron is
\[
 A_{ij}^{kl}=\frac12h_{ijk,l}r_{ijk}+\frac12h_{ijl,k}r_{ijl}.
\]
Geometrically, this is the signed area of the piece of the dual face to the edge $ij$ inside the tetrahedron, decomposed into the two triangular contributions adjacent to the faces $\{i,j,k\}$ and $\{i,j,l\}$.

Glickenstein's derivative formulas identify this signed dual area with the coefficient appearing in the curvature variation \cite{Glickenstein2005}:
\[
 \frac{\partial\alpha_{ijkl}}{\partial r_j}r_ir_j=\frac{A_{ij}^{kl}}{l_{ij}},
\]
and hence
\[
 \mu_{ij}=\frac1{l_{ij}}\sum_{\{i,j,k,l\}\in\T_3}A_{ij}^{kl}.
\]

The following proposition records the sign of the Euclidean weights under condition I and the uniform upper bound that follows from a positive margin.

\begin{proposition}\label{prop:E-mu-positive}
Let $ij\in\T_1$ and assume that, for every tetrahedron $\{i,j,k,l\}$ containing $ij$ and every ordered vertex $a\in\{i,j,k,l\}$,
\[
 x_a\le x_b+x_c+x_d,\qquad \{a,b,c,d\}=\{i,j,k,l\}.
\]
Then every such tetrahedron is real and $\mu_{ij}>0$.  Moreover, if $\T$ has degree bounded by $M$ and there exists $\varepsilon>0$ such that
\[
 x_b+x_c+x_d-x_a\ge\varepsilon,\qquad \{a,b,c,d\}=\{i,j,k,l\},
\]
for every ordered tetrahedron, then
\[
 0<\mu_{ij}\le C(M,\varepsilon)
\]
for every edge $ij\in\T_1$.
\end{proposition}
\begin{proof}
For one tetrahedron $\{i,j,k,l\}$ containing the edge $ij$, Glickenstein's signed-height computation \cite{Glickenstein2005} gives
\[
 \frac{A_{ij}^{kl}}{l_{ij}}
 =\frac{\rho_{ijkl}}{4l_{ij}}
 \left[
 r_{ijk}^2(x_i+x_j+x_k-x_l)
 +r_{ijl}^2(x_i+x_j+x_l-x_k)
 \right],
\]
where $\rho_{ijkl}$ is the radius of the circumscripted sphere.  By Proposition \ref{prop:E-strong-real}, the condition I inequalities imply that every incident tetrahedron is real.  The two factors
\[
 x_i+x_j+x_k-x_l,\qquad x_i+x_j+x_l-x_k
\]
are non-negative, and they cannot vanish simultaneously.  Hence $A_{ij}^{kl}/l_{ij}>0$ for every incident tetrahedron, and therefore $\mu_{ij}>0$ by the preceding formula for $\mu_{ij}$.

It remains to prove the uniform upper bound.  The elementary identities
\[
 r_{ijk}^2=\frac1{x_ix_j+x_ix_k+x_jx_k},
 \qquad
 l_{ij}=\frac{x_i+x_j}{x_ix_j}
\]
imply
\[
 \frac{r_{ijk}^2(x_i+x_j+x_k-x_l)}{l_{ij}}
 \le \frac{x_i+x_j+x_k-x_l}{x_i+x_j}.
\]
Under the positive margin assumption, the inequality $x_i+x_j+x_l-x_k\ge\varepsilon>0$ gives
\[
 x_i+x_j+x_k-x_l<2(x_i+x_j),
\]
and hence the last expression is bounded by $2$.  The same argument applied to the face $\{i,j,l\}$ gives the same bound for the second term in the bracket.  Thus
\[
 0<\frac{A_{ij}^{kl}}{l_{ij}}\le \rho_{ijkl}.
\]
On the other hand,
\[
 Q^E_{ijkl}
 =\sum_{a\in\{i,j,k,l\}}x_a\bigl(x_i+x_j+x_k+x_l-2x_a\bigr)
 \ge \varepsilon(x_i+x_j+x_k+x_l).
\]
Since the four quantities $x_i+x_j+x_k+x_l-2x_a$ have sum $2(x_i+x_j+x_k+x_l)$, the positive margin assumption also gives
\[
 x_i+x_j+x_k+x_l\ge 2\varepsilon.
\]
Consequently $Q^E_{ijkl}\ge2\varepsilon^2$ and
\[
 \rho_{ijkl}=\frac2{\sqrt{Q^E_{ijkl}}}\le \frac{\sqrt2}{\varepsilon}.
\]
Summing over at most $D_M$ tetrahedra containing the edge $ij$ gives the desired bound.
\end{proof}

\subsection{Short-Time Existence of the Euclidean Flow}
We now prove the short-time existence result in the Euclidean background geometry.

\begin{proof}[\textbf{Proof of Theorem \ref{thm:intro-short}(1)}]
Let $M$ be a degree bound for $\T$, and let $\rho>0$ be as in Theorem \ref{thm:intro-short}(1).
Following the finite Dirichlet approximation for infinite combinatorial Yamabe flows \cite{Ji2025}, choose an exhaustion
$V_1\subset V_2\subset\cdots\subset\T_0$ by finite vertex sets, and denote by $\operatorname{Int}(V_n)$ and $\partial V_n$ the corresponding interior and boundary.  Since $\T$ is locally finite, every fixed vertex belongs to $\operatorname{Int}(V_n)$ for all sufficiently large $n$.  On $V_n$ solve the finite Dirichlet problem
\[
\begin{cases}
 \displaystyle\frac{d u_i^{(n)}}{dt}=-K_i(u^{(n)}),
     & i\in\operatorname{Int}(V_n),\\[4pt]
 u_i^{(n)}(0)=0,
     & i\in V_n,\\
 u_i^{(n)}(t)=0,
     & i\in\partial V_n .
\end{cases}
\]
We also extend $u^{(n)}$ to $\T_0$ by setting $u_j^{(n)}(t)=0$ for $j\notin V_n$.  The curvature $K_i(u^{(n)})$ is computed from the metric $r_j^{(n)}(t)=e^{u_j^{(n)}(t)}r_j^0$; for $i\in\operatorname{Int}(V_n)$ it depends only on vertices in $V_n$.  Since only finitely many variables are free and the metric is initially real, the finite-dimensional ODE has a local solution.

By Lemma \ref{lem:curv-bound} and the degree bound,
\[
 |K_i|\le C_M:=4\pi+2\pi D_M
\]
as long as the finite solution is real.  Hence every finite solution satisfies
\[
 |u_i^{(n)}(t)|\le C_Mt
\]
for all vertices $i$, with the convention $u_i^{(n)}=0$ outside $V_n$.  Choose $t_0>0$ such that $C_Mt_0<\rho/2$.  Then $\|u^{(n)}(t)\|_{\ell^\infty}<\rho/2$ on $[0,t_0]$, so the metric $e^{u^{(n)}(t)}r^0$ remains real on this interval.  The finite trajectory therefore stays in a compact subset of the real region.  By the standard continuation theorem for ODEs, the solution extends to all of $[0,t_0]$.

We next record the estimate needed to extract a $C^1$ limit.  If $\|u\|_{\ell^\infty}\le\rho/2$, then every further perturbation $w$ with $\|w\|_{\ell^\infty}\le\rho/2$ satisfies $\|u+w\|_{\ell^\infty}\le\rho$, and hence $e^{u+w}r^0$ is real.  For a tetrahedron $\tau$, set
\[
 S_\tau=\sum_{a\in\tau}x_a,\qquad p_a=\frac{x_a}{S_\tau},\qquad
 q_\tau=\frac{Q^E_\tau}{S_\tau^2}=1-2\sum_{a\in\tau}p_a^2 .
\]
The logarithmic real-neighborhood condition keeps the normalized shapes
$p=(p_a)_{a\in\tau}$ in a compact subset of the real shape space.  Thus there exists $q_0=q_0(\rho)>0$ such that $q_\tau\ge q_0$ for every tetrahedron and every such $u$.

For $\tau=\{i,j,k,l\}$, since $\partial r_j/\partial u_j=r_j$, Glickenstein's dual-area formula used above gives
\[
 \left|\frac{\partial\alpha_{ijkl}}{\partial u_j}\right|
 =\left|\frac{\partial\alpha_{ijkl}}{\partial r_j}r_j\right|
 =\left|\frac{A_{ij}^{kl}}{r_il_{ij}}\right|.
\]
Using the formula for $A_{ij}^{kl}$, we have
\[
\begin{aligned}
\left|\frac{A_{ij}^{kl}}{r_il_{ij}}\right|
&\le \frac{\rho_{ijkl}}{4}
\left[
\frac{r_{ijk}^2}{r_il_{ij}}\left|x_i+x_j+x_k-x_l\right|
+\frac{r_{ijl}^2}{r_il_{ij}}\left|x_i+x_j+x_l-x_k\right|
\right]  \\
&= \frac{\rho_{ijkl}}{4}
\left[
\frac{r_{ijk}^2}{r_il_{ij}}\left|S_\tau-2x_l\right|
+\frac{r_{ijl}^2}{r_il_{ij}}\left|S_\tau-2x_k\right|
\right].
\end{aligned}
\]
Moreover,
\[
 \rho_{ijkl}=\frac2{\sqrt{Q^E_\tau}}\le\frac2{S_\tau\sqrt{q_0}},\qquad
 |S_\tau-2x_a|\le S_\tau,\qquad
 \frac{r_{ijk}^2}{r_il_{ij}}\le1,
\]
and the same estimate holds with $k$ replaced by $l$.  Therefore
\[
\left|\frac{A_{ij}^{kl}}{r_il_{ij}}\right|
\le \frac{\rho_{ijkl}}{4}(S_\tau+S_\tau)
\le \frac{1}{\sqrt{q_0}}.
\]
Thus
\[
 \left|\frac{\partial\alpha_{ijkl}}{\partial u_j}\right|\le C(\rho).
\]
By Lemma \ref{lem:E-schlafli-solid},
\[
 \sum_{a\in\tau}\frac{\partial\alpha_{ijkl}}{\partial u_a}=0.
\]
Hence the diagonal derivative is bounded by the off-diagonal ones.
Summing over the at most $D_M$ tetrahedra incident to a vertex gives
\[
 \left|\frac{\partial K_i}{\partial u_a}(u)\right|\le C(M,\rho)
\]
whenever $\|u\|_{\ell^\infty}\le\rho/2$.  For $i\in\operatorname{Int}(V_n)$,
\[
 (u_i^{(n)})''
 =-\sum_a\frac{\partial K_i}{\partial u_a}(u^{(n)})(u_a^{(n)})',
\]
where the sum is over the closed tetrahedral star of $i$.  Since $|(u_a^{(n)})'|\le C_M$, after changing the constant we obtain
\[
 |(u_i^{(n)})''|\le C(M,\rho)
\]
on $[0,t_0]$.  Boundary and exterior vertices are fixed, so the same bound is trivial there.

For each fixed vertex $i$, the functions $u_i^{(n)}$ are therefore uniformly bounded in $C^2([0,t_0])$ for all sufficiently large $n$.  By Arzela--Ascoli and a diagonal argument over the vertices, a subsequence converges in $C^1([0,t_0])$ at every vertex to a function $u_i(t)$.  The bound above passes to the limit, so $\|u(t)\|_{\ell^\infty}<\rho/2$ on $[0,t_0]$ and the limiting metric remains real.

It remains to pass the equation to the limit. For all sufficiently large $n$, the closed tetrahedral star of $i$ is contained in $V_n$, and hence
\[
 u_i^{(n)}(t)=-\int_0^t K_i(u^{(n)}(s))\,ds .
\]
The curvature $K_i$ depends only on finitely many variables in this star and is smooth on the real region.  Since these finitely many coordinates of $u^{(n)}$ converge uniformly to those of $u$, we may pass to the limit in the integral and obtain
\[
 u_i(t)=-\int_0^t K_i(u(s))\,ds .
\]
Thus $u_i\in C^1([0,t_0))$ and $u_i'=-K_i(u)$ for every vertex $i$.  Since each $K_i$ is a smooth function of finitely many real local variables, repeated differentiation of the local system shows that $u_i$ is smooth in time.
\end{proof}

\subsection{Uniqueness of the Euclidean Flow}
The uniqueness argument reduces the difference of two solutions to a linear parabolic equation on the infinite graph.  The next estimate verifies the sign and bounded weighted-degree hypotheses needed for the maximum principle.

\begin{lemma}[Euclidean coefficient estimate]\label{lem:Epos}
Assume that the Euclidean strong non-degenerate condition II in Definition \ref{def:intro-strong} holds with constants $\varepsilon,C$.  Then, for every edge $ij$, $\mu_{ij}\ge0$, and
\[
 \frac1{r_i}\sum_{j:j\sim i}\mu_{ij}\le C(M,\varepsilon,C).
\]
\end{lemma}
\begin{proof}
By Proposition \ref{prop:E-mu-positive}, all incident tetrahedra are real and $\mu_{ij}\ge0$.  It remains to record the above estimate needed for the maximum principle.

For one tetrahedron $\{i,j,k,l\}$ containing the edge $ij$, the same dual-area formula used in Proposition \ref{prop:E-mu-positive} gives
\[
 \frac{A_{ij}^{kl}}{r_i l_{ij}}
 =\frac{\rho_{ijkl}}{4r_i l_{ij}}
 \left[
 r_{ijk}^2(x_i+x_j+x_k-x_l)
 +r_{ijl}^2(x_i+x_j+x_l-x_k)
 \right].
\]
Under condition II, $x_i+x_j+x_k-x_l$ and $x_i+x_j+x_l-x_k$ are non-negative and bounded above by $C$.  Moreover,
\[
 \frac{r_{ijk}^2}{r_i l_{ij}}
 =\frac{x_i^2x_j}
 {(x_i+x_j)(x_ix_j+x_ix_k+x_jx_k)}
 \le1,
\]
and the same estimate holds with $k$ replaced by $l$.  The positive lower margin in condition II also gives the uniform lower bound $Q^E_{ijkl}\ge2\varepsilon^2$, as shown in Proposition \ref{prop:E-mu-positive}; hence $\rho_{ijkl}=2/\sqrt{Q^E_{ijkl}}$ is uniformly bounded by a constant depending only on $\varepsilon$.  Consequently
\[
 \frac{A_{ij}^{kl}}{r_i l_{ij}}\ge0,\qquad
 \frac{A_{ij}^{kl}}{r_i l_{ij}}\le C(\varepsilon,C)
\]
for each tetrahedral contribution.  Since at most $D_M$ tetrahedra are incident to a fixed vertex and each such tetrahedron contributes to at most three edge terms adjacent to $i$, summing over all neighbors $j$ gives
\[
 \frac1{r_i}\sum_{j:j\sim i}\mu_{ij}
 =\sum_{j:j\sim i}\sum_{\{i,j,k,l\}\in\T_3}
 \frac{A_{ij}^{kl}}{r_i l_{ij}}
 \le C(M,\varepsilon,C).
\]
\end{proof}

We now prove uniqueness of the infinite Euclidean flow.
\begin{proof}[\textbf{Proof of Theorem \ref{thm:intro-unique}(1)}]
Put $u=\log r$, $\hat u=\log\hat r$, and $g=u-\hat u$.  By Newton--Leibniz and Proposition \ref{prop:Eevol},
\[
\begin{aligned}
 K_i(u(t))-K_i(\hat u(t))
 &=\int_0^1\frac{d}{ds}K_i(u^s(t))\,ds  \\
 &=-\int_0^1\frac1{r_i^s(t)}
   \sum_{j:j\sim i}\mu_{ij}(u^s(t))(g_j(t)-g_i(t))\,ds,
\end{aligned}
\]
where $u^s=s u+(1-s)\hat u$ and $r_i^s=e^{u_i^s}$.  Hence
\[
 g_i'=-(K_i(u)-K_i(\hat u))
     =\sum_{j:j\sim i}\omega_{ij}(t)(g_j-g_i),
\]
where
\[
 \omega_{ij}(t)=\int_0^1 \frac{\mu_{ij}(u^s(t))}{r_i^s(t)}\,ds .
\]
By assumption, the interpolating metrics $r^s(t)$ satisfy condition II with the same constants.  This holds for all $s\in[0,1]$ and $t\in[0,T]$.
Lemma \ref{lem:Epos} gives $\omega_{ij}\ge0$ and
\[
 \sum_{j:j\sim i}\omega_{ij}(t)
 =\int_0^1 \frac1{r_i^s(t)}\sum_{j:j\sim i}\mu_{ij}(u^s(t))\,ds
 \le C(M,\varepsilon,C).
\]
Moreover, $g$ is bounded on $\T_0\times[0,T]$.  Indeed, $g_i(0)=0$ and Lemma \ref{lem:curv-bound} gives $|K_i|,|\hat K_i|\le C_M$, so
\[
 |g_i(t)|\le \int_0^t |K_i(u(s))-K_i(\hat u(s))|\,ds\le 2C_MT .
\]
Applying Lemma \ref{lem:mp} with $a_i\equiv1$ and $h_i\equiv0$ to $g$ and to $-g$ gives $g\equiv0$.  Therefore $u=\hat u$ and hence $r=\hat r$.
\end{proof}

\section{Hyperbolic Background Geometry}
\subsection{Hyperbolic Ball-Packings and Curvature Evolution}
In the hyperbolic background geometry, a ball-packing metric assigns a radius $r_i>0$ to each vertex and edge length $l_{ij}=r_i+r_j$ to each edge.  Following the finite-dimensional hyperbolic ball-packing theory \cite{CooperRivin1996,GeHua2020}, set $c_i=\coth r_i$.  For a tetrahedron $\{i,j,k,l\}$, the six lengths $l_{ab}=r_a+r_b$ determine a non-degenerate hyperbolic tetrahedron precisely when
\begin{equation}\label{eq:QH}
 Q^H_{ijkl}=(c_i+c_j+c_k+c_l)^2-2(c_i^2+c_j^2+c_k^2+c_l^2)+4>0.
\end{equation}

\begin{proposition}[Strong non-degeneracy implies hyperbolic reality]\label{prop:H-strong-real}
Let $\tau=\{i,j,k,l\}$ be a hyperbolic ball-packing tetrahedron and set $S_\tau=c_i+c_j+c_k+c_l$.  If
\[
 c_a\le S_\tau-c_a,\qquad a\in\{i,j,k,l\},
\]
then $Q^H_{ijkl}\ge4>0$, and hence $\tau$ is real.
\end{proposition}
\begin{proof}
Since
\[
 Q^H_{ijkl}=\sum_{a\in\tau}c_a(S_\tau-2c_a)+4,
\]
the inequalities $S_\tau-2c_a\ge0$ imply $Q^H_{ijkl}\ge4>0$.
\end{proof}

Thus the hyperbolic strong non-degenerate condition I in Definition \ref{def:intro-strong} is a sufficient condition for all tetrahedra to be real.
The curvature is again
\[
 K_i=4\pi-\sum_{\{i,j,k,l\}\in\T_3}\alpha_{ijkl},
\]
and the hyperbolic Yamabe flow is
\begin{equation}\label{eq:Hflow}
 r_i'=-K_i\sinh r_i.
\end{equation}
With $w_i=\log\tanh(r_i/2)$, this is $w_i'=-K_i$, since $dr_i/dw_i=\sinh r_i$.

Define
\[
 \mu_{ij}=\sum_{\{i,j,k,l\}\in\T_3}\frac{\partial\alpha_{ijkl}}{\partial r_j}\sinh r_i\sinh r_j
\]
and
\[
 h_i=\sum_{\{i,j,k,l\}\in\T_3}\sum_{a\in\{i,j,k,l\}}\frac{\partial\alpha_{ijkl}}{\partial r_a}\sinh r_a.
\]

\begin{proposition}[Hyperbolic curvature evolution]\label{prop:Hevol}
Let $r(t)$ be a smooth family of real hyperbolic ball-packing metrics, and set $w_i(t)=\log\tanh(r_i(t)/2)$.  Then, for every vertex $i$,
\[
 K_i'=-\frac1{\sinh r_i}\sum_{j:j\sim i}\mu_{ij}(w_j'-w_i')-h_iw_i'.
\]
Along \eqref{eq:Hflow},
\begin{equation}\label{eq:Hheat}
 K_i'=\frac1{\sinh r_i}\sum_{j:j\sim i}\mu_{ij}(K_j-K_i)+h_iK_i.
\end{equation}
\end{proposition}
\begin{proof}
For such a family, $dr_a/dw_a=\sinh r_a$, and hence
\[
 r_a'=\sinh r_a\,w_a' .
\]
Fix a tetrahedron $\tau=\{i,j,k,l\}$ incident to $i$.  By the chain rule,
\[
\begin{aligned}
 \frac{d}{dt}\alpha_{ijkl}
 &=\sum_{a\in\{i,j,k,l\}}
   \frac{\partial\alpha_{ijkl}}{\partial r_a}\sinh r_a\,w_a'  \\
 &=\sum_{a\in\{j,k,l\}}
   \frac{\partial\alpha_{ijkl}}{\partial r_a}\sinh r_a(w_a'-w_i')
   +\sum_{a\in\{i,j,k,l\}}
   \frac{\partial\alpha_{ijkl}}{\partial r_a}\sinh r_a\,w_i'  \\
 &=\frac1{\sinh r_i}\sum_{a\in\{j,k,l\}}
   \frac{\partial\alpha_{ijkl}}{\partial r_a}\sinh r_i\sinh r_a(w_a'-w_i')
   +\sum_{a\in\{i,j,k,l\}}
   \frac{\partial\alpha_{ijkl}}{\partial r_a}\sinh r_a\,w_i' .
\end{aligned}
\]
Differentiating $K_i=4\pi-\sum\alpha_{ijkl}$ and summing over all tetrahedra incident to $i$ give
\[
\begin{aligned}
 K_i'
 &=-\frac1{\sinh r_i}
   \sum_{\{i,j,k,l\}\in\T_3}
   \left[
   \frac{\partial\alpha_{ijkl}}{\partial r_j}\sinh r_i\sinh r_j(w_j'-w_i')
   +\frac{\partial\alpha_{ijkl}}{\partial r_k}\sinh r_i\sinh r_k(w_k'-w_i')\right.\\
 &\hspace{4.9cm}\left.
   +\frac{\partial\alpha_{ijkl}}{\partial r_l}\sinh r_i\sinh r_l(w_l'-w_i')
   \right] \\
 &\quad
   -\sum_{\{i,j,k,l\}\in\T_3}
   \sum_{a\in\{i,j,k,l\}}
   \frac{\partial\alpha_{ijkl}}{\partial r_a}\sinh r_a\,w_i' .
\end{aligned}
\]
We now group the first double sum according to the vertex paired with $i$.  The coefficient of $w_j'-w_i'$ is
\[
 \sum_{\{i,j,k,l\}\in\T_3}
 \frac{\partial\alpha_{ijkl}}{\partial r_j}\sinh r_i\sinh r_j
 =\mu_{ij},
\]
where the sum is over tetrahedra containing the edge $ij$.  The coefficient of $w_i'$ in the second term is exactly $h_i$.  Therefore
\[
 K_i'=-\frac1{\sinh r_i}\sum_{j:j\sim i}\mu_{ij}(w_j'-w_i')-h_iw_i' .
\]
Compared with the Euclidean curvature evolution, the additional term involving $h_i$ arises from the volume term in the hyperbolic Schlaefli formula \cite{GeHua2020}.  Along \eqref{eq:Hflow}, one has $w_i'=-K_i$.  Hence
\[
 w_j'-w_i'=(-K_j)-(-K_i)=K_i-K_j,\qquad -h_iw_i'=h_iK_i,
\]
and the preceding formula becomes \eqref{eq:Hheat}.
\end{proof}

\begin{lemma}\label{lem:Hbounds}
Assume that $\T$ has degree bounded by $M$.  Fix a constant $r_*>0$.  If a hyperbolic ball-packing metric satisfies
\[
 r_a\ge r_*\quad\text{for all }a\in\T_0,\qquad
 c_i\le c_j+c_k+c_l
\]
for every ordered tetrahedron $\{i,j,k,l\}\in\T_3$, then all tetrahedra are real and there exists a constant $C=C(M,r_*)$ such that
\[
 \mu_{ij}\ge0,\qquad |h_i|\le C,
\]
and
\[
 \frac1{\sinh r_i}\sum_{j:j\sim i}\mu_{ij}\le C.
\]
\end{lemma}
\begin{proof}
For one tetrahedron $\tau=\{i,j,k,l\}$, write $y_a=c_a=\coth r_a$.  The assumptions imply
\[
 1\le y_a\le c_*:=\coth r_*,\qquad S_\tau-2y_a\ge0 .
\]
By Proposition \ref{prop:H-strong-real}, every tetrahedron is real and $Q^H_\tau\ge4$.

We first prove the sign of the off-diagonal coefficients.  The hyperbolic solid-angle derivative formula of Ge--Hua \cite{GeHua2020} gives, for $j\ne i$,
\[
\frac{\partial\alpha_{ijkl}}{\partial r_j}
=P_{ij}^{kl}B_{ij}^{kl},
\]
where
\[
B_{ij}^{kl}
=2-(y_k-y_l)^2+y_i(y_j+y_k+y_l)+y_j(y_i+y_k+y_l),
\]
and
\[
P_{ij}^{kl}
=\frac{\sinh r_k\sinh r_l}
{\sqrt{Q^H_\tau}\sinh(r_i+r_j+r_k)\sinh(r_i+r_j+r_l)}>0 .
\]
Put
\[
 p=y_i+y_j,\qquad q=y_k+y_l,\qquad d=|y_k-y_l|.
\]
The inequalities $y_k\le y_i+y_j+y_l$ and $y_l\le y_i+y_j+y_k$ give $d\le p$, while $y_k,y_l\ge1$ gives $d\le q-2$.  Then
\[
 B_{ij}^{kl}=2+2y_iy_j+pq-d^2 .
\]
If $q\ge p$, then $d^2\le p^2\le pq$, so $B_{ij}^{kl}\ge2+2y_iy_j>0$.  If $q<p$, then
\[
 B_{ij}^{kl}\ge 2+2y_iy_j+pq-(q-2)^2 .
\]
As a function of $q\in[2,p]$, the right-hand side is concave.  Its values at $q=2$ and $q=p$ are respectively
\[
 2+2y_iy_j+2p,\qquad 2y_iy_j+4p-2,
\]
which are positive because $y_i,y_j\ge1$.  Thus $B_{ij}^{kl}>0$ in this case as well.  Hence $\partial\alpha_{ijkl}/\partial r_j>0$, and consequently $\mu_{ij}\ge0$.

We now prove the uniform bounds.  Since the variables $y_a$ stay in the compact interval $[1,c_*]$ and $Q^H_\tau$ is uniformly bounded away from zero, $B_{ij}^{kl}$ is uniformly bounded.  To estimate the hyperbolic sine factors, we use
\begin{equation}\label{eq:H-sinh-three}
 \sinh(x+y+z)\ge4\sinh x\sinh y\sinh z,\qquad x,y,z>0.
\end{equation}
By \eqref{eq:H-sinh-three} and $r_i\ge r_*$,
\[
 \frac{\sinh r_j\sinh r_k\sinh r_l}
 {\sinh(r_i+r_j+r_k)\sinh(r_i+r_j+r_l)}
 \le C(r_*),
\]
so each off-diagonal contribution
\[
 \left|\frac{\partial\alpha_{ijkl}}{\partial r_j}\sinh r_j\right|
\]
is bounded by a constant depending only on $r_*$.  For the diagonal term, the derivative formula in \cite{GeHua2020} can be written as
\[
 \frac{\partial\alpha_{ijkl}}{\partial r_i}\sinh r_i
 =-\frac{R_i^{jkl}}{\sqrt{Q^H_\tau}}\,
 \Phi(y_i,y_j,y_k,y_l,Q^H_\tau),
\]
where $\Phi$ is uniformly bounded under the same assumptions and
\[
\begin{aligned}
 R_i^{jkl}
 &=
 \frac{\sinh^2 r_i\sinh^2 r_j\sinh^2 r_k\sinh^2 r_l}
 {\sinh(r_i+r_j+r_k)\sinh(r_i+r_j+r_l)\sinh(r_i+r_k+r_l)} .
\end{aligned}
\]
By \eqref{eq:H-sinh-three}, the hyperbolic sine ratio $R_i^{jkl}$ is uniformly bounded.
Therefore
\[
 \left|\frac{\partial\alpha_{ijkl}}{\partial r_a}\sinh r_a\right|\le C(r_*),
 \qquad a\in\{i,j,k,l\},
\]
for every incident tetrahedron.  Summing over at most $D_M$ tetrahedra incident to $i$ gives $|h_i|\le C(M,r_*)$, while
\[
 \frac1{\sinh r_i}\sum_{j:j\sim i}\mu_{ij}
 =\sum_{j:j\sim i}\sum_{\{i,j,k,l\}\in\T_3}
 \frac{\partial\alpha_{ijkl}}{\partial r_j}\sinh r_j
\]
is bounded by the same single-tetrahedron estimate summed over the at most $3D_M$ corresponding terms.
\end{proof}

\subsection{Short-Time Existence of the Hyperbolic Flow}
We now prove the short-time existence result in the hyperbolic background geometry.

\begin{proof}[\textbf{Proof of Theorem \ref{thm:intro-short}(2)}]
Let $M$ be a degree bound for $\T$, and let $\rho>0$ be as in Theorem \ref{thm:intro-short}(2).  Write $w_i^0=\log\tanh(r_i^0/2)$ and use the relative variables
\[
 \xi_i=w_i-w_i^0 .
\]
By assumption, if $\|\xi\|_{\ell^\infty}\le\rho$, then $w^0+\xi<0$ componentwise and the metric
\[
 r_i=2\operatorname{arctanh}(e^{w_i^0+\xi_i})
\]
is real.

Following the finite Dirichlet approximation used above, choose an exhaustion $V_n\nearrow\T_0$ and solve the finite Dirichlet problems
\[
\begin{cases}
 \displaystyle\frac{d\xi_i^{(n)}}{dt}=-K_i(w^0+\xi^{(n)}),
     & i\in\operatorname{Int}(V_n),\\[4pt]
 \xi_i^{(n)}(0)=0,
     & i\in V_n,\\
 \xi_i^{(n)}(t)=0,
     & i\in\partial V_n .
\end{cases}
\]
We extend $\xi^{(n)}$ by zero outside $V_n$.  Since only finitely many variables are free and the initial metric is real, the finite-dimensional ODE has a local solution.  By Lemma \ref{lem:curv-bound},
\[
 |K_i|\le C_M:=4\pi+2\pi D_M
\]
as long as the finite solution is real.  Write $w^{(n)}=w^0+\xi^{(n)}$.  Choose $t_0>0$ so that $C_Mt_0<\rho/2$.  Then
\[
 |\xi_i^{(n)}(t)|\le C_Mt<\rho/2
\]
on $[0,t_0]$.  Thus the finite trajectories stay in the real region and extend to $[0,t_0]$.

It remains to obtain the compactness needed for a $C^1$ limit.  Fix a vertex $i$ and let $S_i$ be the finite set of vertices contained in tetrahedra incident to $i$.  On the compact set determined by $|\xi_a|\le\rho/2$ for $a\in S_i$, the local $w$-coordinates stay in a compact subset of $(-\infty,0)$, and the incident tetrahedra remain in the real hyperbolic region because this set is contained in the real neighborhood $\|\xi\|_{\ell^\infty}\le\rho$.  The solid angles are smooth functions of the local $w$-variables there, so for some local constant $L_i^H<\infty$,
\[
 \left|\frac{\partial K_i}{\partial w_a}\right|\le L_i^H,\qquad a\in S_i .
\]
For all sufficiently large $n$ with $i\in\operatorname{Int}(V_n)$,
\[
 (\xi_i^{(n)})''
 =-\sum_{a\in S_i}\frac{\partial K_i}{\partial w_a}(w^{(n)})(\xi_a^{(n)})',
\]
and $|(\xi_a^{(n)})'|\le C_M$.  Therefore
\[
 |(\xi_i^{(n)})''|\le C_i^H
\]
on $[0,t_0]$, with $C_i^H$ independent of $n$.  Arzela--Ascoli and a diagonal argument give a subsequence converging in $C^1([0,t_0])$ at every vertex to a limit $\xi$.  Since each curvature depends only on finitely many local variables and is smooth on the real region, we may pass to the limit in the integral form of the finite equations and obtain
\[
 \xi_i(t)=-\int_0^t K_i(w^0+\xi(s))\,ds .
\]
Thus $w_i=w_i^0+\xi_i$ satisfies $w_i'=-K_i$ on $[0,t_0]$.  Returning to
\[
 r_i=2\operatorname{arctanh}(e^{w_i})
\]
gives a solution of \eqref{eq:Hflow}.  Since each $K_i$ depends only on finitely many local variables and is smooth in the real region, the limiting equation is a local finite-dimensional smooth ODE; repeated differentiation gives smoothness in time.
\end{proof}

\subsection{Uniqueness of the Hyperbolic Flow}
The uniqueness proof is again an application of the infinite-graph maximum principle, now with the zeroth-order term in the hyperbolic curvature evolution.

\begin{proof}[\textbf{Proof of Theorem \ref{thm:intro-unique}(2)}]
Put $g_i(t)=w_i(t)-\hat w_i(t)$.  For $s\in[0,1]$, set
\[
 w^s=sw+(1-s)\hat w,\qquad
 r_i^s=2\operatorname{arctanh}(e^{w_i^s}).
\]
For fixed $t$, the path $s\mapsto w^s(t)$ consists of real hyperbolic ball-packing metrics by Lemma \ref{lem:Hbounds}, and
\[
 \frac{d}{ds}w_i^s(t)=g_i(t).
\]
Applying Proposition \ref{prop:Hevol} with $s$ as the path parameter gives
\[
 \frac{d}{ds}K_i(w^s(t))
 =-\frac1{\sinh r_i^s(t)}
   \sum_{j:j\sim i}\mu_{ij}(w^s(t))(g_j(t)-g_i(t))
   -h_i(w^s(t))g_i(t).
\]
Hence Newton--Leibniz gives
\[
\begin{aligned}
 K_i(w(t))-K_i(\hat w(t))
 &=\int_0^1\frac{d}{ds}K_i(w^s(t))\,ds\\
 &=-\int_0^1\frac1{\sinh r_i^s(t)}
   \sum_{j:j\sim i}\mu_{ij}(w^s(t))(g_j(t)-g_i(t))\,ds\\
 &\quad-\int_0^1 h_i(w^s(t))g_i(t)\,ds .
\end{aligned}
\]
Since $g_i'=-(K_i(w)-K_i(\hat w))$, we obtain
\[
 g_i'=\sum_{j:j\sim i}\omega_{ij}(t)(g_j-g_i)+q_i(t)g_i,
\]
where
\[
 \omega_{ij}(t)=\int_0^1
 \frac{\mu_{ij}(w^s(t))}{\sinh r_i^s(t)}\,ds,\qquad
 q_i(t)=\int_0^1 h_i(w^s(t))\,ds .
\]
The lower bound and the hyperbolic strong non-degenerate condition I in Theorem \ref{thm:intro-unique}(2) are precisely the hypotheses needed to apply Lemma \ref{lem:Hbounds} uniformly to the interpolating metrics $r^s(t)$.  Therefore $\omega_{ij}\ge0$ and
\[
 \sum_{j:j\sim i}\omega_{ij}(t)\le C,\qquad q_i(t)\le C.
\]
Moreover, $g$ is bounded on $\T_0\times[0,T]$.  Indeed, $g_i(0)=0$ and Lemma \ref{lem:curv-bound} gives $|K_i|,|\hat K_i|\le C_M$, so
\[
 |g_i(t)|\le \int_0^t |K_i(w(s))-K_i(\hat w(s))|\,ds\le 2C_MT .
\]
With $a_i\equiv1$ and the weights $\omega_{ij}(t)$ above, this is an equation of the form $g'=\Delta g+qg$.  Applying Lemma \ref{lem:mp} to $g$ and to $-g$ gives $g\equiv0$.  Therefore $w=\hat w$ and hence $r=\hat r$.
\end{proof}

\section{Long-Time Existence of Extended Flows}
We now prove the global existence results for the extended flows.  The exhaustion and compactness argument follows the standard strategy for infinite combinatorial flows \cite{Ji2025}: finite Dirichlet problems give uniform $C^0$ and Lipschitz bounds in the appropriate conformal variable, Arzela--Ascoli gives a locally uniform limit, and the following elementary regularity lemma identifies the weak time derivative.

\begin{lemma}\label{lem:weak-c1}
Let $f$ be Lipschitz on $[0,T]$ and let $g$ be continuous on $[0,T]$.  If
\[
 \int_0^T f(t)\phi'(t)\,dt=-\int_0^T g(t)\phi(t)\,dt
\]
for every smooth $\phi$ with $\phi(0)=\phi(T)=0$, then $f\in C^1([0,T])$ and $f'=g$.
\end{lemma}
The proof can be found in \cite[Section 5.8]{Evans2010}.

\subsection{The Extended Euclidean Flow}
In Euclidean background geometry the extended flow is
\begin{equation}\label{eq:Eext}
 r_i'=-\wt K_i r_i.
\end{equation}

\begin{proof}[\textbf{Proof of Theorem \ref{thm:intro-extended}(1)}]
Write
\[
 u_i(t)=\log\frac{r_i(t)}{r_i^0}.
\]
Then \eqref{eq:Eext} is equivalent to $u_i'=-\wt K_i(u)$, where $\wt K_i(u)$ is computed from the metric $r_i^0e^{u_i}$.  Choose a finite exhaustion $V_n\nearrow\T_0$.  On $V_n$ solve the finite Dirichlet problem
\[
\begin{cases}
 \displaystyle\frac{d u_i^{(n)}}{dt}=-\wt K_i(u^{(n)}),
     & i\in\operatorname{Int}(V_n),\\[4pt]
 u_i^{(n)}(0)=0,
     & i\in V_n,\\
 u_i^{(n)}(t)=0,
     & i\in\partial V_n ,
\end{cases}
\]
and extend $u^{(n)}$ by zero outside $V_n$.  The extended solid angles are continuous on the positive radius region, so the finite-dimensional vector field is continuous.  Peano's theorem gives a local solution.

For a fixed finite set $V_n$, Lemma \ref{lem:curv-bound} applies to the extended solid angles and gives
\[
 |\wt K_i|\le 4\pi+2\pi d_i,\qquad i\in\operatorname{Int}(V_n).
\]
Hence the finite trajectory is bounded on every finite time interval on which it exists.  Thus, for each fixed vertex $i$ and all sufficiently large $n$,
\[
 |u_i^{(n)}(t)|+|(u_i^{(n)})'(t)|\le C(i,T)
\]
on every finite time interval $[0,T]$.

By the preceding uniform bound and Lipschitz bound, Arzela--Ascoli and a diagonal argument give a subsequence, still denoted by $u^{(n)}$, such that for every vertex $i$ and every $T>0$,
\[
 u_i^{(n)}\to u_i
\]
uniformly on $[0,T]$.  The limit $u_i$ is Lipschitz on $[0,T]$.  It remains to identify its time derivative.

Fix $i$ and choose $n$ large enough that $i\in\operatorname{Int}(V_n)$.  For every smooth $\phi$ with $\phi(0)=\phi(T)=0$,
\[
 \int_0^T u_i^{(n)}\phi'\,dt
 =\int_0^T \wt K_i(u^{(n)})\phi\,dt .
\]
The curvature $\wt K_i$ depends only on the finite tetrahedral star of $i$ and is continuous there.  Since the relevant coordinates of $u^{(n)}$ converge uniformly, dominated convergence gives
\[
 \int_0^T u_i\phi'\,dt
 =\int_0^T \wt K_i(u)\phi\,dt .
\]
Applying Lemma \ref{lem:weak-c1} with $g=-\wt K_i(u)$ gives $u_i\in C^1([0,T])$ and $u_i'=-\wt K_i(u)$.  Since $T$ and $i$ were arbitrary, $u\in C_t^1(\T_0\times[0,\infty))$.  Returning to $r_i(t)=r_i^0e^{u_i(t)}$ gives a global solution of \eqref{eq:Eext}.
\end{proof}

\subsection{The Extended Hyperbolic Flow}
In hyperbolic background geometry the extended flow is
\begin{equation}\label{eq:Hext}
 r_i'=-\wt K_i\sinh r_i.
\end{equation}
Equivalently, in the variable $w_i=\log\tanh(r_i/2)$, it is $w_i'=-\wt K_i$.

\begin{lemma}[Large-radius estimate for extended hyperbolic solid angles]\label{lem:H-large-angle}
For every $\delta>0$, there exists $R(\delta)>0$ such that, for any real or virtual hyperbolic tetrahedron $\tau$ and any vertex $i\in\tau$, if
\[
 r_i\ge R(\delta),\qquad r_i=\max_{a\in\tau}r_a,
\]
then $\wt\alpha_{i,\tau}\le\delta$.
\end{lemma}
This is the large-radius estimate for extended hyperbolic solid angles proved in \cite{GeHua2020}.

\begin{proposition}[Upper bound for finite hyperbolic approximations]\label{prop:Hupper}
Suppose that $\T$ has degree bounded by $M$.  Let $r^0:\T_0\to(0,\infty)$ satisfy $\sup_i r_i^0\le N$.  For a finite set $V\subset\T_0$ and $T>0$, let $r:\T_0\times[0,T]\to(0,\infty)$ be obtained by solving the finite Dirichlet problem
\[
\begin{cases}
\displaystyle \frac{d r_i}{dt}=-\wt K_i(r(t))\sinh r_i(t),
    & i\in\operatorname{Int}(V),\\[4pt]
r_i(0)=r_i^0,
    & i\in V,\\
r_i(t)=r_i^0,
    & i\in\partial V,
\end{cases}
\]
where outside $V$ we set $r_i(t)=r_i^0$.  Then
\[
 r_i(t)\le C(M,N)
\]
for all $(i,t)\in\T_0\times[0,T]$.
\end{proposition}
\begin{proof}
Choose $\delta>0$ so that $D_M\delta\le2\pi$, and let $R(\delta)$ be given by Lemma \ref{lem:H-large-angle}.  Set
\[
 R_0=\max\{N,R(\delta)\}.
\]
We claim that no finite Dirichlet solution can cross the level $R_0$.  Indeed, suppose that it first crosses $R_0+\eta$ for some $\eta>0$ at time $t_0$, and choose a vertex $i\in V$ with maximal radius at $t_0$.  Since boundary vertices and vertices outside $V$ keep their initial radii, this vertex lies in $\operatorname{Int}(V)$.  Moreover, $i$ is a maximal-radius vertex in every incident tetrahedron.  Hence every extended solid angle at $i$ is at most $\delta$, and, since $d_i\le D_M$,
\[
 \sum_{\tau\ni i}\wt\alpha_{i,\tau}\le D_M\delta\le2\pi .
\]
Thus $\wt K_i\ge2\pi$ at $(i,t_0)$, and
\[
 r_i'=-\wt K_i\sinh r_i<0.
\]
This contradicts the first-crossing property.  Letting $\eta\downarrow0$ gives the desired bound with $C(M,N)=R_0$.
\end{proof}

\begin{proof}[\textbf{Proof of Theorem \ref{thm:intro-extended}(2)}]
Let $w_i^0=\log\tanh(r_i^0/2)$ and write $\zeta_i=w_i-w_i^0$.  In this variable the extended hyperbolic flow becomes
\[
 \zeta_i'=-\wt K_i(\zeta),
\]
where the curvature is computed from $r_i=2\operatorname{arctanh}(e^{w_i^0+\zeta_i})$.  Choose a finite exhaustion $V_n\nearrow\T_0$ and solve
\[
\begin{cases}
 \displaystyle\frac{d\zeta_i^{(n)}}{dt}=-\wt K_i(\zeta^{(n)}),
     & i\in\operatorname{Int}(V_n),\\[4pt]
 \zeta_i^{(n)}(0)=0,
     & i\in V_n,\\
 \zeta_i^{(n)}(t)=0,
     & i\in\partial V_n ,
\end{cases}
\]
extending $\zeta^{(n)}$ by zero outside $V_n$.  The finite-dimensional vector field is continuous, so Peano's theorem gives a local solution.

By Lemma \ref{lem:curv-bound} and the degree bound,
\[
 |\wt K_i|\le C_M:=4\pi+2\pi D_M.
\]
Hence
\[
 |(\zeta_i^{(n)})'(t)|\le C_M,\qquad
 |\zeta_i^{(n)}(t)|\le C_Mt .
\]
On each finite Dirichlet problem this prevents $r_i$ from reaching zero in finite time, because $w_i=w_i^0+\zeta_i$ stays finite from below on every finite interval.  Proposition \ref{prop:Hupper} prevents the radii from blowing up, and hence keeps $w_i$ away from $0$.  Therefore the finite trajectory remains in a compact subset of the domain $w_i<0$ on every finite time interval, and the finite Dirichlet solutions extend for all time.

On every finite time interval $[0,T]$, the functions $\zeta_i^{(n)}$ are uniformly bounded and uniformly Lipschitz for each fixed vertex $i$.  Arzela--Ascoli and a diagonal argument give a subsequence converging uniformly on $[0,T]$ at every vertex to a Lipschitz function $\zeta_i$.

Fix $i$ and choose $n$ large enough that $i\in\operatorname{Int}(V_n)$.  For every smooth $\phi$ with $\phi(0)=\phi(T)=0$,
\[
 \int_0^T \zeta_i^{(n)}\phi'\,dt
 =\int_0^T \wt K_i(\zeta^{(n)})\phi\,dt .
\]
The extended curvature at $i$ depends only on the finite tetrahedral star of $i$ and is continuous for positive radii.  For this fixed local star, the lower bounds from $w_a^0+\zeta_a^{(n)}\ge w_a^0-C_MT$ and the upper bound in Proposition \ref{prop:Hupper} keep all relevant radii in a compact subset of the positive radius region on $[0,T]$.  Since the relevant coordinates of $\zeta^{(n)}$ converge uniformly, dominated convergence gives
\[
 \int_0^T \zeta_i\phi'\,dt
 =\int_0^T \wt K_i(\zeta)\phi\,dt .
\]
By Lemma \ref{lem:weak-c1}, $\zeta_i\in C^1([0,T])$ and $\zeta_i'=-\wt K_i(\zeta)$.  Since $T$ and $i$ were arbitrary, returning to
\[
 r_i(t)=2\operatorname{arctanh}(e^{w_i^0+\zeta_i(t)})
\]
gives a global $C_t^1$ solution of \eqref{eq:Hext}.
\end{proof}

\noindent\textbf{Acknowledgements.}
I am deeply grateful to my supervisor, Professor Bobo Hua, for his constant encouragement, generous support, and insightful guidance throughout this research.

\end{document}